\documentclass{article}
\pdfoutput=1
\usepackage[english]{babel}
\usepackage{amsmath,amssymb,bbm,xcolor,ntheorem}
\usepackage[framemethod=tikz]{mdframed}
\usepackage[twoside]{geometry}
\usepackage[T1]{fontenc}
\usepackage{seqsplit}
\usepackage{mathptmx}
\usepackage[activate={true,nocompatibility},final,tracking=true,kerning=true,spacing=true,factor=1100,stretch=10,shrink=10]{microtype}
\usepackage{fancyhdr}
\fancyhfoffset[LE]{6mm}
\fancyhfoffset[RO]{6mm}
\newcommand{\Eta}{\mathrm{H}}
\newcommand{\Mu}{\mathrm{M}}
\newcommand{\TeXmacs}{T\kern-.1667em\lower.5ex\hbox{E}\kern-.125emX\kern-.1em\lower.5ex\hbox{\textsc{m\kern-.05ema\kern-.125emc\kern-.05ems}}}
\newcommand{\assign}{:=}

\newcommand{\nin}{\not\in}
\newcommand{\tmem}[1]{{\em #1\/}}
\newcommand{\tmname}[1]{\textsc{#1}}
\newcommand{\tmop}[1]{\ensuremath{\operatorname{#1}}}
\newcommand{\tmrsup}[1]{\textsuperscript{#1}}
\newcommand{\tmstrong}[1]{\textbf{#1}}
\newcommand{\tmtextit}[1]{{\itshape{#1}}}
\newcommand{\tmtextsl}[1]{{\slshape{#1}}}
\newcommand{\withTeXmacstext}{This document has been produced using \TeXmacs \ (see \texttt{http://www.texmacs.org})}
\definecolor{grey}{rgb}{0.75,0.75,0.75}
\definecolor{orange}{rgb}{1.0,0.501,0.0}
\definecolor{brown}{rgb}{0.5,0.25,0.0}
\definecolor{pink}{rgb}{1.0,0.5,0.5}
\definecolor{dark blue}{rgb}{0.0,0.0,0.501}
\definecolor{dark green}{rgb}{0.0,0.501,0.0}
\definecolor{dark cyan}{rgb}{0.0,0.501,0.501}
\definecolor{pastel cyan}{rgb}{0.874,1.0,1.0}
\definecolor{pastel yellow}{rgb}{1.0,1.0,0.874}
\definecolor{grey blue}{rgb}{0.666,0.666,1.0}
\definecolor{red brown}{rgb}{0.666,0.0,0.0}

\newenvironment{itemizearrow}{\begin{itemize} }{\end{itemize}}

\newenvironment{proof}{\medskip\noindent\textbf{Proof\ }}{\hspace*{\fill}$\Box$\medskip}

{\theorembodyfont{\slshape}\newtheorem*{remark*}{Remark}}
{\theorembodyfont{\slshape}\newtheorem*{commentary*}{Commentary}}

\mdfsetup{linecolor=black,linewidth=0.5pt,skipabove=0.5em,skipbelow=0.5em,hidealllines=false,
innerleftmargin=5pt,innerrightmargin=5pt,innertopmargin=5pt,innerbottommargin=5pt}

\newmdenv[hidealllines=false,innertopmargin=1ex,innerbottommargin=1ex,innerleftmargin=1ex,innerrightmargin=1ex]{tmframed}

\mdfdefinestyle{thmstyle}{linewidth=1pt,frametitlerule=false}

\newmdtheoremenv[style=thmstyle,linecolor=dark blue, fontcolor=dark blue]{defi}{\color{dark blue}Definition}

\newmdtheoremenv[style=thmstyle,linecolor=dark green, fontcolor=dark green]{twier}{\color{dark green}Theorem}

\newmdtheoremenv[style=thmstyle, fontcolor=dark blue]{notation}{\color{dark blue}Notation}
\renewtheorem*{notation*}{Notation}
\surroundwithmdframed[innerleftmargin=0pt,innerrightmargin=0pt,hidealllines=true,fontcolor=dark blue]{notation*}

{\theorembodyfont{\rmfamily}\newtheorem{example}{Example}}
\surroundwithmdframed[innerleftmargin=0pt,innerrightmargin=0pt,hidealllines=true,fontcolor=dark cyan]{example}

\newtheorem{lemma}{Lemma}
\surroundwithmdframed[innerleftmargin=0pt,innerrightmargin=0pt,hidealllines=true,fontcolor=red brown]{lemma}

\newmdtheoremenv[style=thmstyle, fontcolor=dark green,innerleftmargin=0pt,innerrightmargin=0pt,hidealllines=true]{proposition}{\color{dark green}Proposition}

\newmdtheoremenv[style=thmstyle,outerlinewidth=2pt,fontcolor=orange,innerrightmargin=1pt,hidealllines=true,leftline = true,font=\bfseries\boldmath]{corollary}{\color{orange}Corollary}

\newtheorem*{question*}{Question}
\surroundwithmdframed[innerleftmargin=0pt,innerrightmargin=0pt,hidealllines=false,linecolor=grey,outerlinewidth=1.25pt,
backgroundcolor= yellow!40!red!40]{question*}

\tikzstyle{titzdata}=[draw=gray, thick, fill=pastel yellow,
  text=black, rectangle,rounded corners, minimum height=0.5cm]

\begin{document}

\title{The {\L}ojasiewicz Exponent of\\
Semiquasihomogeneous
Singularities\thanks{{\withTeXmacstext}} \thanks{MSC 2010: 32S05, 58K05\enspace Keywords: {\L}ojasiewicz exponent, quasihomogeneous singularity}}
\author{\tmname{Szymon Brzostowski}}
\maketitle

\begin{mdframed}[outerlinewidth = 0.75pt ,roundcorner = 5 pt,backgroundcolor=pastel cyan,innertopmargin=1.2\baselineskip,
skipabove={\dimexpr0.5\baselineskip+\topskip\relax},
skipbelow={1em},
singleextra={\node[titzdata,xshift=0.5\mdfboundingboxwidth,xshift=-0.25mm] at (P-|O) {\sc\bfseries Abstract};},
linecolor=gray]
\indent\indent Let $f: ( \mathbbm{C}^{n} ,0 ) \rightarrow ( \mathbbm{C},0
)$ be a semiquasihomogeneous function. We give a formula for the local
{\L}ojasiewicz exponent $\mathcal{L}_{0} ( f )$ of $f$, in terms of weights of
$f$. In particular, in the case of a quasihomogeneous isolated singularity
$f$, we generalize a formula for $\mathcal{L}_{0} ( f )$ of
{\tmname{Krasi{\'n}ski}}, {\tmname{Oleksik}} and {\tmname{P{\l}oski}}
({\cite{KOP09}}) from $3$ to $n$ dimensions. This was previously announced in
{\cite{TYZ10}}, but as a matter of fact it has not been proved correctly
there, as noticed by the {\tmname{AMS}} reviewer {\tmname{T.~Krasi{\'n}ski}}.

As a consequence of our result, we get that the {\L}ojasiewicz exponent is
invariant in topologically trivial families of singularities coming from a
quasihomogeneous germ. This is an affirmative partial answer to Teissier's
conjecture.
\end{mdframed}

\section{Introduction}

The (local) {\L}ojasiewicz exponent $l_{0} ( f )$ of a holomorphic map-germ
$f: ( \mathbbm{C}^{n} ,0 ) \rightarrow ( \mathbbm{C}^{m} ,0 )$ is one of the
possible generalizations of the order function from $1$ to $n$ indeterminates.
Namely, for $n=1$ the condition $\tmop{ord}  f \assign \inf   \{ \tmop{ord} 
f_{i} \} =k$ is equivalent to the condition
\begin{equation}
  C_{1}  | z |^{k} \leqslant \| f ( z ) \| \leqslant C_{2}  | z |^{k} ,
  \label{w1}
\end{equation}
for some positive constants $C_{1}$, $C_{2}$, in a neighbourhood of $0$. (Here
$k< \infty$ exactly when $f \neq 0$ i.e. when $f$ is a finite mapping.) For
$n>1$, the above condition splits: the optimal exponent in the right
inequality (that is the biggest one) leads to the ordinary notion of the order
$\tmop{ord}  f$ of $f$ while the optimal exponent in the left inequality (that
is the smallest one) is the {\L}ojasiewicz exponent $l_{0} ( f )$ of $f$.
(Note that $l_{0} ( f ) < \infty$ exactly when $f$ is a finite mapping, unlike
the order function.) To state it more formally,
\[ l_{0} ( f ) \assign \inf   \left\{ q \in \mathbbm{Q}_{>0} : \; C_{1}  \| z
   \|^{q} \leqslant \| f ( z ) \| , \text{ for some } C_{1} >0 \text{ and all
   } \| z \| \ll 1 \right\} . \]
The norms $\| \cdot \|$ in the definition are any convenient ones. An
important (and, at first, maybe a little surprising) fact is that $l_{0} ( f
)$, if finite, is a positive {\tmem{rational}} number. Indeed, one can prove
that in such a case the defining {\tmem{infimum}} is in fact the
{\tmem{minimum}}. Moreover, there exist holomorphic curves $0 \neq \varphi : (
\mathbbm{C},0 ) \rightarrow ( \mathbbm{C}^{n} ,0 )$ satisfying
\begin{equation}
  l_{0} ( f ) = \tfrac{\tmop{ord}  f \circ \varphi}{\tmop{ord}   \varphi} ,
  \label{w2}
\end{equation}
and such curves are optimal in the sense that for every $0 \neq \psi : (
\mathbbm{C},0 ) \rightarrow ( \mathbbm{C}^{n} ,0 )$ it holds $l_{0} ( f )
\geqslant \tfrac{\tmop{ord}  f \circ \psi}{\tmop{ord}   \psi}$. (For the
proofs of these facts, see {\cite{LT08}} or {\cite{Plo88}}.) It easily follows
that, for the optimal curves $\varphi$,
\[ \| f \circ \varphi \| \sim \| \varphi \|^{l_{0} ( f )} , \]
so in a way we recover the inequalities (\ref{w1}) that we started from.

Much is known about $l_{0} ( f )$ for $n=2$ (see
{\cite{LT08,KL77,Tei77a,CK88,Plo88,CK95,Len98,Plo01}}), the case $n \geqslant
3$ is however challenging still (for some recent results see e.g.
{\cite{Biv09,KOP09,Ole09,Plo10,BKO12,BE12}}).

In this paper we are interested in the {\tmem{{\L}ojasiewicz exponent of a
singularity}} $f: ( \mathbbm{C}^{n} ,0 ) \rightarrow ( \mathbbm{C},0 )$,
\begin{equation}
  \mathcal{L}_{0} ( f ) \assign l_{0} \left( \tfrac{\partial f}{\partial
  z_{1}} , \ldots , \tfrac{\partial f}{\partial z_{n}} \right) .
\end{equation}
More specifically, we give formulas for $\mathcal{L}_{0} ( f )$ if $f$ is a
(weakly) (semi-)quasihomogeneous singularity (see Definition \ref{de1}), in
terms of its weights. Such formulas are known when $n \leqslant 3$
({\cite{KOP09}} for the quasihomogeneous, and {\cite{BKO12}} for
the~semiquasihomogeneous case). In {\cite{TYZ10}}, there appeared an
{\tmem{incorrect}} proof of an analogous formula for general $n$
(cf.~{\tmname{AMS}} review MR2679619 for the details). We aim at~giving a
valid proof of this result (Theorem \ref{th3}).

Our approach to the problem follows closely that of {\cite{KOP09}}: first we
prove a general lemma which is interesting in its own right (Proposition
\ref{prop1}), and then~apply~it to deal with the non-generic situations for
the computation of the {\L}ojasiewicz exponent. After proving Theorem
\ref{th3}, we pass to the more general situation allowing ,,weak'' weights.
Here, most of the necessary ingredients are delivered by {\tmname{Saito}}
({\cite{Sai71}}) whose results allow us to reduce the general problem to the
semiquasihomogeneous one using stable equivalence (Corollary \ref{cor2.5} and
Theorem \ref{th4}) and also express the value of the {\L}ojasiewicz exponent
of a quasihomogeneous singularity in a coordinate-independent fashion (Theorem
\ref{th5}). We conclude with the
observation~that~{\tmem{Teissier's~conjecture}} (see at the end of
the paper) is valid in the class of weakly semiquasihomogeneous functions
(Corollary \ref{cor4}).

The main results of the paper are Theorems \ref{th3}--\ref{th5} and Corollary
\ref{cor4}.

\section{Definitions and Known Facts}

{\noindent}The following definitions are much in the spirit of
{\tmname{Saito}} {\cite{Sai71}} and {\tmname{Arnold}} {\cite{Arn74}}.

\begin{defi}\label{de1}
Let $f: ( \mathbbm{C}^{n} ,0 ) \rightarrow ( \mathbbm{C},0
)$ be a (germ of a) holomorphic function. Then:
\begin{itemizearrow}
  \item $f$ is an {\tmem{(isolated) singularity}}, if it has an isolated
  critical point at $0$
  
  \item $f$ is {\tmem{weakly quasihomogeneous of type $( d;l_{1} , \ldots
  ,l_{n} )$}}, shortly: {$f \in$\tmname{WQH}$( d;l_{1} , \ldots ,l_{n} )$}, if
  $l= ( l_{1} , \ldots ,l_{n} ) \in \mathbbm{R}^{n}$, $d \in \mathbbm{R}_{>0}$
  and for every monomial~$z^{a} =z_{1}^{a_{1}}\cdot \ldots \cdot z_{n}^{a_{n}}$ appearing
  in the expansion of $f$ with a non-zero coefficient it holds $\langle a,l
  \rangle \assign a_{1} l_{1} + \ldots +a_{n} l_{n} =d$; in particular, $f=0$
  is {\tmname{wqh}} of all types
\end{itemizearrow}
The numbers $l_{1} , \ldots ,l_{n}$ will be called {\tmem{weights}}. The
number $\deg_{l} ( z^{a} ) := \langle a,l \rangle$ will be referred to as the
{\tmem{weighted degree}} of a monomial $z^{a}$. For a series $g ( z ) =
\sum_{a \in \mathbbm{N}^{n}_{0}} g_{a} z^{a}$, its {\tmem{weighted order}} is
$\tmop{ord}_{l}  g \assign \underset{g_{a} \neq 0}{\inf} ( \deg_{l} ( z^{a} )
)$.
\begin{itemizearrow}
  \item $f$ is {\tmem{quasihomogeneous of type $( d;l_{1} , \ldots ,l_{n}
  )$}}, shortly: $f \in${\tmname{QH}}$( d;l_{1} , \ldots ,l_{n} )$, if it is
  {\tmname{wqh}} of type $( d;l_{1} , \ldots ,\allowbreak l_{n} )$ with $d,l_{1} , \ldots
  ,l_{n} \in \mathbbm{Q}$ and $l_{1} /d, \ldots ,l_{n} /d \in ( 0,1/2 ]$
\end{itemizearrow}

   Note that from the definition it follows that a
  {\tmname{qh}} $f$ is necessarily a (germ of a) polynomial of order greater
  than or equal to $2$.
  \begin{itemizearrow}
    \item $f$ is {\tmem{weakly semiquasihomogeneous of type $( d;l_{1} ,
    \ldots ,l_{n} )$}}, shortly: $f \in${\tmname{WSQH}}$( d;l_{1} , \ldots
    ,l_{n} )$, if it can be written in the form $f=f_{0} +f'$, where $f_{0}$
    is a {\tmname{wqh}} {\tmem{singularity}} of type $( d;l_{1} , \ldots
    ,l_{n} )$, $\tmop{ord}  f' >1$ and every monomial appearing in the
    expansion of $f'$ is of weighted degree greater than $d$
  \end{itemizearrow}
  The singularity $f_{0}$ will be called the {\tmem{principal part of $f$}}.
  \begin{itemizearrow}
    \item $f$ is {\tmem{semiquasihomogeneous of type $( d;l_{1} , \ldots
    ,l_{n} )$}}, shortly: $f \in${\tmname{SQH}}$( d;l_{1} , \ldots ,l_{n} )$,
    if it is {\tmname{wsqh}} of type $( d;l_{1} , \ldots ,l_{n} )$ with
    $d,l_{1} , \ldots ,l_{n} \in \mathbbm{Q}$ and $l_{1} /d, \ldots ,
    l_{n} /d \in ( 0,1/2 ]$ (or $f_{0}$ is a {\tmname{qh}}
    singularity)
  \end{itemizearrow}
  It is known that a {\tmname{sqh}} $f$ (or a {\tmname{wsqh}} $f$ with
  positive weights) is automatically an isolated singularity {\cite{Arn74}}.
\end{defi}

\tmtextsl{
\begin{remark*}
It is easy to see that the types $( d;l_{1} , \ldots
,l_{n} )$ in the definitions can always be normalized to $( 1;l_{1} /d, \ldots
,\allowbreak l_{n} /d )$. Saito, \tmtextit{op.~cit.}, allowed also complex weights for
{\tmname{wqh}} functions, however, as he proved, it is often not restrictive
to consider only the rational ones. On the other hand, Arnold,
\tmtextit{op.~cit.}, considered mostly {\tmname{sqh}} functions; the
definition of {\tmname{wsqh}} functions is perhaps somewhat non-standard.
\end{remark*}}


The following theorem holds {\cite[Thm.~1, Cor.~4 and Thm.~3]{KOP09}}.

\begin{twier}\label{th1}
Let $f: ( \mathbbm{C}^{n} ,0 ) \rightarrow ( \mathbbm{C},0
)$, where $n \leqslant 3$, be a weakly quasihomogeneous singularity of type $(
1;l_{1} , \ldots ,\allowbreak l_{n} )$ with positive rational weights. Put $w_{i} \assign
1/l_{i}$. Then
\begin{equation}
  \mathcal{L}_{0} ( f ) = \min \left( \max_{1 \leqslant i \leqslant n} ( w_{i}
  -1 ) , \underset{1 \leqslant i \leqslant n}{\prod} ( w_{i} -1 ) \right) .
  \label{w4}
\end{equation}
{\noindent}In particular, if $f$ is quasihomogeneous
\[ \mathcal{L}_{0} ( f ) = \max_{1 \leqslant i \leqslant n} ( w_{i} -1 ) .
\]
\end{twier}

\begin{remark*}
Actually, formula (\ref{w4}) is proved in {\cite{KOP09}}
only for $n=3$. However, for a function $f$ of \ $2$ indeterminates one can
consider the function $\tilde{f} \assign f+z_{3}^{2}$, which has the same
{\L}ojasiewicz exponent as $f$ and for which the weight $l_{3} =1/2$, and then
apply formula (\ref{w4}) to it to find an analogous formula for
$\mathcal{L}_{0} ( f )$.
\end{remark*}

Theorem \ref{th1} is known to generalize to the case of a {\tmname{sqh}}
function $f$ ({\cite[Theorem 3.2]{BKO12}}) in exactly the same form. Namely,
taking account of the remark above, one can state:

\begin{twier}\label{th2}
Let $f: ( \mathbbm{C}^{n} ,0 ) \rightarrow ( \mathbbm{C},0
)$, where $n \leqslant 3$, be a weakly semiquasihomogeneous function with
positive rational weights and principal part $f_{0}$. Then
\[ \mathcal{L}_{0} ( f ) =\mathcal{L}_{0} ( f_{0} ) . \]
\end{twier}

\section{Results}

We begin with a proof of a proposition which is a weaker version of \
{\cite[Thm.~2]{KOP09}}, but generalized to $n \geqslant 4$ indeterminates. We
remark that in {\tmem{loc.~cit.}} the theorem is stated as a very special case
of local Hilbert's Nullstellensatz and the authors conjecture it to be true in
any dimension ({\tmem{op.~cit.}}, Problem 1). It is however not the case,
already for four indeterminates (cf.~Example \ref{ex1}).

\begin{notation*}
For a germ $f \in \mathcal{O}^{n}$ of $n$
indeterminates and $i \in \{ 1, \ldots ,n \}$ we define $\nabla f \assign
\left( \frac{\partial f}{\partial z_{1}} , \ldots , \frac{\partial f}{\partial
z_{n}} \right)$ and $\text{}^{\hat{i}} \nabla f \assign \left( \frac{\partial
f}{\partial z_{1}} , \ldots , \widehat{\frac{\partial f}{\partial z_{i}}} ,
\ldots , \frac{\partial f}{\partial z_{n}} \right)$, where the hat means
omission. For a set $F$ of germs, $\mathcal{V} ( F )$ will denote the germ at
$0$ of the set of common zeroes of the system $F$.
\end{notation*}

\begin{proposition}
  \label{prop1}Let $f: ( \mathbbm{C}^{n} ,0 ) \rightarrow ( \mathbbm{C},0 )$
  be an isolated singularity such that $\mathcal{V} ( \text{}^{\hat{1}} \nabla
  f ) \subset \mathcal{V} ( z_{1} )$. Then $\tmop{ord}  f=2$. Moreover, there
  exists an $1<i \leqslant n$ such that the monomial $z_{1}z_{i}$ does appear
  in the expansion of $f$ with a non-zero coefficient while the monomial
  $z_{i}^{2}$ does not.
\end{proposition}

\begin{proof}
  Let us consider the deformation $f_{s} ( z ) \assign f ( z ) +s\,z_{1}^{2}$ of
  the germ $f$. For any $s \in \mathbbm{C}$, it is $\mathcal{V} (
  \text{}^{\hat{1}} \nabla f_{s} ) =\mathcal{V} ( \text{}^{\hat{1}} \nabla f )
  =:\mathcal{A}$ because in fact $( \text{}^{\hat{1}} \nabla f_{s} )
  \mathbbm{C} \{ z_{1} , \ldots ,z_{n} \} = ( \text{}^{\hat{1}}
  \nabla f ) \mathbbm{C} \{ z_{1} , \ldots ,z_{n} \}$ as ideals. Take a set
  $\Phi$ of non-equivalent parametrizations of the curve $\mathcal{A}$. Then
  one has
  \[ \mu ( f_{s} ) = \sum_{\varphi \in \Phi} l_{\varphi} \tmop{ord}   \left(
     \frac{\partial f_{s}}{\partial z_{1}} \circ \varphi \right) , \]
  where the numbers $l_{\varphi}$ are the multiplicities of the branches
  $\varphi$ of the curve $\mathcal{A}$. But by assumption, $\varphi \subset
  \mathcal{V} ( \text{}^{\hat{1}} \nabla f ) \subset \mathcal{V} ( z_{1} )$
  and hence $\tmop{ord}   \left( \frac{\partial f_{s}}{\partial z_{1}} \circ
  \varphi \right) = \tmop{ord}   \left( \frac{\partial f}{\partial z_{1}}
  \circ \varphi \right)$, for every $\varphi \in \Phi$. Thus, one has $\mu (
  f_{s} ) = \mu ( f )$, $s \in \mathbbm{C}$, or in another words -- $( f_{s}
  )$ is a $\mu$-constant deformation of $f$. Using a result of
  {\tmname{Trotman}} (see {\cite{Tro80}} or {\cite[Prop.~1.1]{PT12}}), we
  conclude that the family $( f_{s} )$ is equimultiple. Since $f$ is a
  singularity, $\tmop{ord}  f_{s} =2$ for $s$ close to $0$, and hence also
  $\tmop{ord}  f=2$.
  
  Let $q ( z )$ be the quadratic form of $f$. Assume, to the contrary, that
  the form does not depend on $z_{1}$. Then, by the splitting lemma, $f$ can
  be transformed through a biholomorphic change of coordinates $\Psi$ into
  $\tilde{f} = \tilde{h} ( z_{1} , \ldots ,z_{k} ) +z_{k+1}^{2} + \ldots
  +z_{n}^{2}$, where $\tmop{ord}   \tilde{h} \geqslant 3$ and $k \geqslant 1$.
  Moreover, it is easy to see that $\Psi$ can be chosen so that $\Psi ( z ) =
  ( z_{1} , \ldots )$ (by assumption; just recall that $\Psi$ engages
  essentially at most those variables that appear in the form $q ( z )$).
  However, such change of parameters does not drag the zero set of
  $\text{}^{\hat{1}} \nabla f$ out of the hyperplane $z_{1} =0$, i.e.
  $\mathcal{V} ( \text{}^{\hat{1}} \nabla \tilde{f} ) \subset \mathcal{V} (
  z_{1} )$. Thus also $\mathcal{V} ( \text{}^{\hat{1}} \nabla \tilde{h} )
  \subset \mathcal{V} ( z_{1} )$, which is impossible by what we already know.
  It follows that $q$ does depend on $z_{1}$. More precisely, since every $f-
  \vartheta z^{2}_{1}$, $\vartheta \in \mathbbm{C}$, also fulfills the
  conditions of the proposition, we deduce that in the expansion of $f$ there
  has to appear a monomial of the form $z_{1} z_{i}$, $i \neq 1$. If for every
  such $i$, in $q ( z )$ there appeared also the monomial $z^{2}_{i}$, then we
  would be able to transform the function $f$ into one that does not contain
  $z_{1}$ in its quadratic form, possibly except for $z_{1}^{2}$, but still
  one that fulfills the assumptions of the proposition; contradiction.
\end{proof}

\begin{remark*}
For $f$ (semi)quasihomogeneous, instead of Trotman's
theorem one can use the results of {\cite{Gre86}} or {\cite{OS87}} to prove
that $\tmop{ord}  f=2$.
\end{remark*}

  \begin{example}
  \label{ex1}Although we will not prove it, we remark that using Proposition
  \ref{prop1} it is possible to show that every singularity $f: (
  \mathbbm{C}^{n} ,0 ) \rightarrow ( \mathbbm{C},0 )$ satisfying $\mathcal{V}
  ( \text{}^{\hat{1}} \nabla f ) \subset \mathcal{V} ( z_{1} )$ can be
  transformed into the form $f=z_{1} z_{2} +g ( z_{2} , \ldots ,z_{n} ) +h (
  z_{1} )$ by a formal change of coordinates whose first component is the
  identity. Thus, in order to prove that under the assumptions of Proposition
  \ref{prop1} it holds $z_{1} \in ( \text{}^{\hat{1}} \nabla f )$
  (cf.~{\cite[Problem 1]{KOP09}}), it is enough to check this for
  singularities of the form indicated above. Now, for $n=3$ it turns out that
  $g$ can be further transformed into one that does not depend on $z_{2}$,
  which gives an alternative proof of {\tmem{op.~cit.}} Theorem 2. For $n
  \geqslant 4$, let us consider any singularity $g_{0} =g_{0} ( z_{3} , \ldots
  ,z_{n} )$ such that $g_{0} \nin \left( \frac{\partial g_{0}}{\partial z_{3}}
  , \ldots , \frac{\partial g_{0}}{\partial z_{n}} \right)$, i.e. a $g_{0}$
  that is not quasihomogeneous in any system of coordinates, and put $f
  \assign z_{1} z_{2} + ( 1+z_{2} ) g_{0}$. Assume, to the contrary, that
  $z_{1} \in ( \text{}^{\hat{1}} \nabla f ) \mathbbm{C} \{ z_{1} , \ldots
  ,z_{n} \}$. It is easy to see that there has to exist a relation of the form
  $z_{1} = \frac{\partial f}{\partial z_{2}} +A_{3}  \frac{\partial
  f}{\partial z_{3}} + \ldots +A_{n}  \frac{\partial f}{\partial z_{n}}$,
  where $A_{j} \in \mathbbm{C} \{ z_{2} , \ldots ,z_{n} \}$. But this relation
  implies that $g_{0} \in \left( \frac{\partial f}{\partial z_{3}} , \ldots ,
  \frac{\partial f}{\partial z_{n}} \right) \mathbbm{C} \{ z_{2} , \ldots
  ,z_{n} \}$ and hence -- that also $g_{0} \in \left( \frac{\partial
  g_{0}}{\partial z_{3}} , \ldots , \frac{\partial g_{0}}{\partial z_{n}}
  \right) \mathbbm{C} \{ z_{3} , \ldots ,z_{n} \}$, contradiction. As a more
  specific example, one can consider for instance $f \assign z_{1} z_{2} + (
  1+z_{2} )  ( z_{3}^{4} +z_{3}^{2} z_{4}^{3} +z_{4}^{5} )$. It can be
  checked, using a computer algebra system, that $z_{1} \nin ( \nabla f )$ but
  $z_{1}^{2} \in ( \text{}^{\hat{1}} \nabla f )$.
  
  We suspect that it may be the case that $\mathcal{V} ( \text{}^{\hat{1}}
  \nabla f ) \subset \mathcal{V} ( z_{1} )$ for a singularity $f$ implies
  $z_{1}^{n-2} \in ( \text{}^{\hat{1}} \nabla f )$, for $n \geqslant 3$.
\end{example}

Using Proposition \ref{prop1} we easily deduce the following.

 \begin{corollary}
  For every semiquasihomogeneous function $f$ of type $( d;l_{1} , \ldots
  ,l_{n} )$ such that $0<l_{j} /d<1/2$, $j=1, \ldots ,n$, and every $i \in \{
  1, \ldots ,n \}$ it is
  \[ \mathcal{V} ( \text{}^{\hat{i}} \nabla f )\not\subset \mathcal{V} ( z_{i} ) . \]
\end{corollary}

\begin{proof}
  Since then $f$ is of order greater than $2$.
\end{proof}

  \begin{corollary}
  \label{cor2}Let $f: ( \mathbbm{C}^{n} ,0 ) \rightarrow ( \mathbbm{C},0 )$ be
  a quasihomogeneous singularity of type $( 1;l_{1} , \ldots ,l_{n} )$. Assume
  that $l_{1} \leqslant \ldots \leqslant l_{n}$ and $\mathcal{V} (
  \text{}^{\hat{1}} \nabla f ) \subset \mathcal{V} ( z_{1} )$. Then $f$ is a
  homogeneous polynomial of order $2$. In particular, $\mathcal{L}_{0} ( f ) =1$.
\end{corollary}

\begin{proof}
  By Proposition \ref{prop1}, in $f$ there appears a monomial $z_{1} z_{i}$
  with a non-zero coefficient. It follows that $l_{1} =l_{i} =1/2$ and hence
  also all the other weights are equal to $1/2$. Now we can apply {\cite[Lemme
  2.4]{Plo85}} to conclude that $\mathcal{L}_{0} ( f ) = \underset{1 \leqslant
  j \leqslant n}{\max} \left( \tmop{ord}   \frac{\partial f}{\partial z_{j}}
  \right) =1$.
\end{proof}

\begin{twier}\label{th3}
Let $f: ( \mathbbm{C}^{n} ,0 ) \rightarrow ( \mathbbm{C},0
)$ be a semiquasihomogeneous function of type $( 1;l_{1} , \ldots ,l_{n} )$.
Put $w_{i} \assign 1/l_{i}$. Then
\begin{equation}
  \mathcal{L}_{0} ( f ) = \max_{1 \leqslant i \leqslant n} ( w_{i} -1 ) .
  \label{w5}
\end{equation}
\end{twier}

\begin{proof}
  First assume that $f$ is quasihomogeneous. Let $l_{1} \leqslant \ldots
  \leqslant l_{n}$. If \ $\mathcal{V} ( \text{}^{\hat{1}} \nabla f ) \subset
  \mathcal{V} ( z_{1} )$ then formula (\ref{w5}) is valid by Corollary
  \ref{cor2}. In the opposite case, it its enough to apply {\cite[Proposition
  2]{KOP09}}.
  
  For $f$ semiquasihomogeneous, Corollary 4.8 of {\cite{BE12}} or Proposition
  4.1 of {\cite{BKO12}} assert that $\mathcal{L}_{0} ( f_{0} ) \leqslant
  \mathcal{L}_{0} ( f )$, where $f_{0}$ is the principal part of $f$. In
  another words, \ $\max_{1 \leqslant i \leqslant n} ( w_{i} -1 ) \leqslant
  \mathcal{L}_{0} ( f )$. By {\cite[Proposition 2.2]{Plo85}} we obtain
  the~opposite inequality.
\end{proof}

It remains to consider the case of ,,weak weights''. For this purpose, we
adopt the results of {\cite{Sai71}} to {\tmname{wsqh}} functions. First, we
prove a {\tmname{wsqh}} variant of {\tmname{splitting lemma}}.

  \begin{lemma}
  \label{le1}Let $f: ( \mathbbm{C}^{n} ,0 ) \rightarrow ( \mathbbm{C},0 )$ be
  a weakly semiquasihomogeneous function of type $( 1;l ) \assign (1;
  \underbrace{p_{1} , \ldots ,p_{i}}_{x} ,\allowbreak \underbrace{q_{1} , \ldots
  ,q_{k}}_{y} , \underbrace{r_{1} , \ldots ,r_{i}}_{z} )$, where $2i+k=n$, and
  of the form
  \begin{equation}
    f ( x,y,z ) = (f_{0} ( y ) + \sum_{j=1}^{i} x_{j} z_{j} ) +f' ( x,y,z ) ,
    \label{w5.5}
  \end{equation}
  where $( x,y,z ) \assign ( x_{1} , \ldots ,x_{i} ,y_{1} , \ldots ,y_{k}
  ,z_{1} , \ldots ,z_{i} )$ and $f_{0} ( y ) + \sum_{j=1}^{i} x_{j} z_{j}$ is
  the principal part of $f$. Assume that $q_{1} , \ldots ,q_{k} >0$. Then
  either $f$ is stably equivalent to a weakly semiquasihomogeneous singularity
  of type $( 1;q_{1} , \ldots ,\allowbreak q_{k} )$ having $f_{0} ( y )$ as its principal
  part (if $k>0$) or is of type $\mathcal{A}_{1}$ (if $k=0$).
\end{lemma}

\begin{proof}
  For $k=0$ it is enough to apply the ordinary splitting lemma; hence in the
  following we will assume that $k>0$. Also we exclude the case $i=0$, as it
  is trivial.
  
  First note that by our assumptions it is $p_{j} +r_{j} =1$ ($j=1, \ldots
  ,i$). Let $\Mu \assign \mu ( f_{0} ) < \infty$ and let $g \in
  \text{{\tmname{WSQH}}} ( 1;l )$ be such that $\tmop{ord}_{l}   ( f-g ) >
  \Mu$ and of the form $g ( x,y,z ) = (f_{0} ( y ) + \sum_{j=1}^{i} x_{j}
  z_{j} ) +g' ( x,y,z )$, $(f_{0} ( y ) + \sum_{j=1}^{i} x_{j} z_{j} )$ being
  the principal part of $g$ and $g'$ being a non-zero polynomial. Recall that
  by Definition \ref{de1} necessarily $\tmop{ord}  g>1$. We claim that for
  each $m \in \mathbbm{N}$ it is possible to write $g$ as
  \begin{equation}
    g ( x,y,z ) =f_{0} ( y ) + \Eta^{( m )} ( y ) + \Lambda^{( m )} ( x,y,z )
    , \label{w6.5}
  \end{equation}
  where $\tmop{ord}_{l}   \Eta^{( m )} >1$,
  \begin{equation}
    \Lambda^{( m )} ( x,y,z ) = \sum_{1 \leqslant j \leqslant i} ( x_{j} +
    \Gamma^{( m )}_{j} ( x,y,z ) )  ( z_{j} + \Delta_{j}^{( m )} ( x,y,z ) )
    +R^{( m )} ( x,y,z ) , \label{w7.5}
  \end{equation}
  $\tmop{ord}_{l}   \Gamma^{( m )}_{j} \geqslant d^{( 1 )} -r_{j}$,
  $\tmop{ord}_{l}   \Delta^{( m )}_{j} \geqslant d^{( 1 )} -p_{j}$ $( j=1,
  \ldots ,i )$ and $\tmop{ord}_{l}  R^{( m )} \geqslant m ( d^{( 1 )} -1 ) +1$
  with $d^{( 1 )} \assign \tmop{ord}_{l}  g' \in ( 1, \infty )$. Moreover
  $\Eta^{( m )}$, $\Lambda^{( m )}$, $\Gamma_{j}^{( m )}$, $\Delta^{( m
  )}_{j}$, $R^{( m )}$ are polynomials in $x,y,z$, vanishing at $0$.
  
  Indeed, for $m=1$ it is enough to put $R^{( 1 )} \assign g' ( x,y,z )$ and
  $\Lambda^{( 1 )} \assign \sum_{j=1}^{i} x_{j} z_{j} +R^{( 1 )}$ so that
  $\Gamma^{( 1 )}_{j} \assign \Delta^{( 1 )}_{j} \assign 0$ $( j=1, \ldots ,i
  )$ and $\Eta^{( 1 )} \assign 0$.
  
  Now, assuming (\ref{w6.5}) and (\ref{w7.5}) for some $m \in \mathbbm{N}$, we
  decompose $R^{( m )}$ into {\tmem{polynomials}} in the following way:
  \begin{equation}
    R^{( m )} = \eta^{( m+1 )} ( y ) + \sum_{1 \leqslant j \leqslant i} (
    x_{j} \delta^{( m+1 )}_{j} ( x,y,z ) +z_{j} \gamma^{( m+1 )}_{j} ( x,y,z )
    ) , \label{w8}
  \end{equation}
  where $\tmop{ord}_{l}   \eta^{( m+1 )} >1$ and $x_{j} \delta^{( m+1 )}_{j}$,
  $z_{j} \gamma^{( m+1 )}_{j}$ are of weighted order greater than or equal to
  $\tmop{ord}_{l}  R^{( m )}$. We put $\Eta^{( m+1 )} \assign \Eta^{( m )} +
  \eta^{( m+1 )}$, $\Lambda^{( m+1 )} \assign \Lambda^{( m )} - \eta^{( m+1 )}
  ( y )$, $\Gamma^{( m+1 )}_{j} \assign \Gamma^{( m )}_{j} + \gamma^{( m+1
  )}_{j}$ and $\Delta^{( m+1 )}_{j} \assign \Delta^{( m )}_{j} + \delta^{( m+1
  )}_{j}$ $( j=1, \ldots ,i )$; clearly, these are polynomials, vanishing at
  $0$. Moreover, $\tmop{ord}_{l}^{}   \Eta^{( m+1 )} >1$,
  \begin{eqnarray}
    \tmop{ord}_{l}   \gamma^{( m+1 )}_{j} & \geqslant & \tmop{ord}_{l}  R^{( m
    )} - \tmop{ord}_{l}  z_{j} \geqslant m ( d^{( 1 )} -1 ) +1-r_{j} \geqslant
    d^{( 1 )} -r_{j} ,  \label{w8.5}\\
    \tmop{ord}_{l}   \delta^{( m+1 )}_{j} & \geqslant & \tmop{ord}_{l}  R^{( m
    )} - \tmop{ord}_{l}  x_{j} \geqslant m ( d^{( 1 )} -1 ) +1-p_{j} \geqslant
    d^{( 1 )} -p_{j}  \label{w9.5}
  \end{eqnarray}
  and hence $\tmop{ord}_{l}   \Gamma^{( m+1 )}_{j} \geqslant d^{( 1 )}
  -r_{j}$, $\tmop{ord}_{l}   \Delta^{( m+1 )}_{j} \geqslant d^{( 1 )} -p_{j}$
  $( j=1, \ldots ,i )$. By (\ref{w7.5}) and (\ref{w8}) we have
  \begin{eqnarray*}
    \Lambda^{( m+1 )} & = & \sum_{1 \leqslant j \leqslant i} ( x_{j} +
    \Gamma^{( m )}_{j} )  ( z_{j} + \Delta_{j}^{( m )} ) +R^{( m )} - \eta^{(
    m+1 )} =\\
    & = & \sum_{1 \leqslant j \leqslant i} ( x_{j} + \Gamma^{( m )}_{j} )  (
    z_{j} + \Delta_{j}^{( m )} ) + \sum_{1 \leqslant j \leqslant i} ( x_{j}
    \delta^{( m+1 )}_{j} +z_{j} \gamma^{( m+1 )}_{j} ) =\\
    & = & \sum_{1 \leqslant j \leqslant i} ( x_{j} + \Gamma^{( m+1 )}_{j} ) 
    ( z_{j} + \Delta_{j}^{( m+1 )} ) +R^{( m+1 )} ,
  \end{eqnarray*}
  where $R^{( m+1 )} \assign - \sum_{1 \leqslant j \leqslant i} ( \Gamma^{( m
  )}_{j} \delta^{( m+1 )}_{j} + \gamma^{( m+1 )}_{j} \Delta^{( m+1 )}_{j} )$.
  Using (\ref{w9.5}) and induction hypothesis, $\tmop{ord}_{l}   \Gamma^{( m
  )}_{j} \delta^{( m+1 )}_{j} = \tmop{ord}_{l}   \Gamma^{( m )}_{j} +
  \tmop{ord}_{l}   \delta^{( m+1 )}_{j} \geqslant ( d^{( 1 )} -r_{j} ) + ( m (
  d^{( 1 )} -1 ) +1-p_{j} ) = ( m+1 )  ( d^{( 1 )} -1 ) +1$ $( j=1, \ldots ,i
  )$ and similarly for the other terms. Hence also $\tmop{ord}_{l}  R^{( m+1
  )} \geqslant ( m+1 )  ( d^{( 1 )} -1 ) +1$. Thus, (\ref{w6.5}) and
  (\ref{w7.5}) hold for $m+1$ and -- by induction -- for all $m \in
  \mathbbm{N}$.
  
  Fix any $m \in \mathbbm{N}$ and consider $\tilde{g}_{m} \assign g-R^{( m
  )}$. Since $\tmop{ord}_{l}   \Gamma^{( m )}_{j} \geqslant d^{( 1 )} -r_{j}
  >1-r_{j} =p_{j} = \deg_{l}  x_{j}$ and similarly $\tmop{ord}_{l}   \Delta^{(
  m )}_{j} \geqslant d^{( 1 )} -p_{j} > \deg_{l}  z_{j}$ for $j=1, \ldots ,i$,
  putting $\Psi ( x,y,z ) \assign ( x_{1} + \Gamma^{( m )}_{1} , \ldots ,x_{i}
  + \Gamma^{( m )}_{i} ,y,z_{1} + \Delta^{( m )}_{1} , \ldots ,z_{i} +
  \Delta^{( m )}_{i} )$ we easily see that $\Psi ( 0 ) =0$ and the matrix
  $\left. \frac{\partial \Psi}{\partial ( x,y,z )} \right|_{( x,y,z ) =0}$ is,
  up to permutation of rows, a triangular one. Composing $\tilde{g}_{m}$ with
  $\Psi^{-1}$ and then with a linear transformation, we find by (\ref{w6.5})
  and (\ref{w7.5}) that $\tilde{g}_{m}$ goes into the form
  \begin{equation}
    \breve{g}_{m} ( x,y,z ) :=f_{0} ( y ) + \Eta^{( m )} ( y ) +
    \sum_{j=1}^{i} ( x^{2}_{j} +z_{j}^{2} ) . \label{w10.5}
  \end{equation}
  Since (by our assumptions and Definition \ref{de1}) $f_{0}$ is a
  {\tmname{wqh}} singularity of type $( 1;q ) \assign ( 1;q_{1} , \ldots
  ,q_{k} )$ and $\tmop{ord}_{q}   \Eta^{( m )} \allowbreak = \tmop{ord}_{l}   \Eta^{( m )}
  >1$, we deduce that $\breve{g}_{m}$ is {\tmname{wsqh}} of type $\left( 1;
  \tfrac{1}{2} , \ldots , \frac{1}{2} ,q_{1} , \ldots ,q_{k} , \tfrac{1}{2} ,
  \ldots , \frac{1}{2} \right)$, and since all these weights are
  {\tmem{positive}}, we conclude that $\breve{g}_{m}$ is a singularity. More
  precisely, $\mu ( \breve{g}_{m} ) = \mu ( f_{0} ) =M< \infty$ (by the very
  same reasoning as in {\cite[Theorem 3.1]{Arn74}} and by the
  stable-equivalence-invariancy of the Milnor number). Hence, the degree of
  (right) determinacy of $\breve{g}_{m}$ -- and thus also of $\tilde{g}_{m}$
  -- can be bounded from above by $M$.
  
  Now recall that $\tmop{ord}_{l}  R^{( m )} \geqslant m ( d^{( 1 )} -1 ) +1
  \xrightarrow[m \rightarrow \infty]{} \infty$, because $d^{( 1 )} >1$. Since
  there are only finitely many monomials of given (ordinary) degree, it
  follows that $\overline{\lim}_{m \rightarrow \infty}   \tmop{ord}  R^{( m )}
  = \infty$. Hence, there exists an $m_{0} \in \mathbbm{N}$ for which the
  order of $R^{( m_{0} )}$ is higher than the number $M$. $\tilde{g}_{m_{0}}$
  being $M$-determined, this means that $\tilde{g}_{m_{0}}$ is
  biholomorphically equivalent to $g= \tilde{g}_{m_{0}} +R^{( m_{0} )}$ and
  thus we conclude that also $g$ is $M$-determined and biholomorphically
  equivalent to $\breve{g}_{m_{0}}$. But in such a case $f$ is right
  equivalent to $g$ and hence -- to $\breve{g}_{m_{0}}$. Finally, from
  (\ref{w10.5}) it follows that $f$ is stably equivalent to the
  {\tmname{wsqh}} singularity $f_{0} ( y ) + \Eta^{( m_{0} )} ( y )$ of type
  $( 1;q )$ and of principal part equal to $f_{0}$. The lemma is proved.
\end{proof}

  \begin{corollary}
  \label{cor2.5}Let $f: ( \mathbbm{C}^{n} ,0 ) \rightarrow ( \mathbbm{C},0 )$
  be a weakly semiquasihomogeneous function of type $( 1;l_{1} , \ldots ,l_{n}
  )$. Let $l_{1} \leqslant \ldots \leqslant l_{i} \leqslant 0<l_{i+1}
  \leqslant \ldots \leqslant l_{i+k} <1 \leqslant l_{i+k+1} \leqslant \ldots
  \leqslant l_{n}$. Then $f$ is a singularity and is either stably equivalent
  to a weakly semiquasihomogeneous function of type $( 1;l_{i+1} , \ldots
  ,l_{i+k} )$ if $k \neq 0$, or is of type $\mathcal{A}_{1}$ if $k=0$.
\end{corollary}

\begin{proof}
  By {\cite[Korollar 1.9]{Sai71}}, $i=n-i-k$ and moreover $l_{j} +l_{n+1-j}
  =1$, for $1 \leqslant j \leqslant i$. Repeating the proof of {\cite[Lemma
  1.10]{Sai71}} one concludes that the principal part $f_{0}$ of $f$ can be
  written in the following form
  \begin{equation}
    f_{0} ( x,y,z ) = \tilde{f}_{0} ( y ) + \sum_{j=1}^{i} ( g_{j} ( z )
    +h_{j} ( x,y,z ) ) x_{i+1-j} , \label{w6}
  \end{equation}
  where the coordinates in $( \mathbbm{C}^{n} ,0 )$ are denoted by $( x,y,z )
  \assign ( x_{1} , \ldots ,x_{i} ,y_{i+1} , \ldots ,y_{i+k} ,z_{i+k+1} ,
  \ldots ,z_{n} )$, the map-germ $( g_{1} , \ldots ,g_{i} )$ is a
  biholomorphism of $( \mathbbm{C}^{i} ,0 )$ and the functions $h_{j}$ satisfy
  $h_{j} ( 0,0,z ) =0$. It follows that $G ( x,y,z ) \assign ( x,y,g_{1} ( z )
  +h_{1} ( x,y,z ) , \ldots ,g_{i} ( y ) +h_{i} ( x,y,z ) )$ is a
  biholomorphism of $( \mathbbm{C}^{n} ,0 )$. Moreover, (\ref{w6}) implies
  that each component $G_{j}$ of $G$ is {\tmname{wqh}} of type $( l_{j} ;l_{1}
  , \ldots ,l_{n} )$, for $j=1, \ldots ,n$. Using the identity $\tmop{Id}
  =G^{-1} \circ G$ we easily check that each component $G_{j}^{-1}$ of
  $G^{-1}$ is also {\tmname{wqh}} of type $( l_{j} ;l_{1} , \ldots ,l_{n} )$,
  for $j=1, \ldots ,n$. Hence, for every monomial $w$ the function $( G^{-1}
  )^{\ast}  w$ is {\tmname{wqh}} of type $( \deg_{l}  w;l_{1} , \ldots ,l_{n}
  )$. It follows that $\tilde{f} \assign ( G^{-1} )^{\ast}  f$ is
  {\tmname{wsqh}} of type $( 1;l_{1} , \ldots ,l_{n} )$. Writing $f=f_{0}
  +f'$, we have
  \[ \tilde{f} ( x,y,z ) = ( G^{-1} )^{\ast}  f ( x,y,z ) = ( \tilde{f}_{0} (
     y ) + \sum_{j=1}^{i} z_{i+k+j} x_{i+1-j} ) + ( G^{-1} )^{\ast}  f' (
     x,y,z ) . \]
  Since the weights $l_{i+1} , \ldots ,l_{i+k}$ are positive, we can apply
  Lemma \ref{le1} to $\tilde{f}$. Clearly, this gives the required assertions
  also for $f=G^{\ast}   \tilde{f}$.
\end{proof}

  \begin{corollary}
  \label{cor3}Let $f: ( \mathbbm{C}^{3} ,0 ) \rightarrow ( \mathbbm{C},0 )$ be
  a weakly semiquasihomogeneous function of type $( 1;l_{1} ,l_{2} ,l_{3} )$
  with rational weights. Put $w_{i} \assign \left\{\begin{array}{ll}
    1/l_{i} , & \text{if } l_{i} \nin \{ 0,1 \}\\
    2, & \text{if } l_{i} \in \{ 0,1 \}
  \end{array}\right.$. Then formula $( \ref{w4} )$ holds with $n=3$.
\end{corollary}

\begin{proof}
  Let $l_{1} \leqslant l_{2} \leqslant l_{3}$. Assume that $l_{1} \leqslant
  0$. Again by {\cite[Korollar 1.9]{Sai71}} we conclude that $l_{3} =1-l_{1}$
  and $0<l_{2} <1$. Corollary \ref{cor2.5} asserts that $f$ is stably
  equivalent to a {\tmname{wsqh}} function $\tilde{f}$ of type $( 1;l_{2} )$.
  Hence, $0<l_{2} \leqslant \frac{1}{2}$ and $\mathcal{L}_{0} ( \tilde{f} ) =
  \tmop{ord}   \tilde{f} -1= \frac{1}{l_{2}} -1$. Since the {\L}ojasiewicz
  exponent is an invariant of the stable equivalence, $\mathcal{L}_{0} ( f ) =
  \frac{1}{l_{2}} -1$. We easily check that $\left( \ref{w4} \right)$ has the
  same value. The case of all weights being positive is covered by Theorems
  \ref{th1} and \ref{th2}.
\end{proof}

We remark that the above corollary is not true for $n=2$ unless one modifies
formula (\ref{w4}) slightly. Similarly, it has been shown by M.~S{\k e}kalski
{\cite{Sek10}} that formula (\ref{w4}) of Theorem \ref{th1} fails in
dimensions $n \geqslant 4$, even for {\tmname{wqh}} singularities with
positive rational weights. A version of formula (\ref{w4}), valid in any
dimension, is given below, in Theorem \ref{th4}.

We illustrate Corollary \ref{cor3} with the following example.

  \begin{example}
  Let us consider $f ( x,y,z ) =xy+x^{4} y^{3} + ( z+y )^{3}$. The singularity
  $f$ is {\tmname{wsqh}} of type $( 1;-2,3,1/3 )$ with principal part $f_{0}
  =xy+x^{4} y^{3} +z^{3}$ and also it is {\tmname{wsqh}} of type $(
  1;2/3,1/3,1/3 )$ with principal part $\tilde{f}_{0} =xy+ ( z+y )^{3}$.
  Corollary \ref{cor3} asserts that the {\L}ojasiewicz exponent of $f$ is
  equal to $2$, in both cases.
\end{example}

\begin{twier}\label{th4}
Let $f: ( \mathbbm{C}^{n} ,0 ) \rightarrow ( \mathbbm{C},0
)$ be a weakly semiquasihomogeneous function of type $( 1;l_{1} , \ldots
,l_{n} )$. Define the multisets $L \assign [ l_{1} , \ldots ,l_{n} ]$, $L^{-}
\assign \left[ 1-a:a \in L \wedge a> \frac{1}{2} \right]$ and the set $L^{0}
\assign ( L \setminus L^{-} ) \cup \left\{ \frac{1}{2} \right\}$. Put
$l_{\min} \assign \min  L^{0}$. Then $\mathcal{L}_{0} ( f ) =
\frac{1}{l_{\min}} -1$.
\end{twier}

\begin{proof}
  By {\cite[Korollar 1.9]{Sai71}} the number of the weights $\leqslant 0$ is
  the same as the number of the weights $\geqslant 1$ and such weights are
  distributed symmetrically with respect to $1/2$, counting multiplicities.
  Hence, upon applying Corollary \ref{cor2.5}, we can assume that $0<l_{j}
  <1$, $j=1, \ldots ,n$. (This affects neither $\mathcal{L}_{0} ( f )$ nor the
  number $l_{\min}$, also in the degenerate case of all the weights lying
  outside the interval $( 0,1 )$.)
  
  Now we repeat the last part of the proof of {\cite[Satz 1.3]{Sai71}}. Let
  $0<l_{1} \leqslant \ldots \leqslant l_{n} <1$ and assume that $l_{n} >
  \frac{1}{2}$. This means that the principal part $f_{0} ( z )$ of $f ( z )$
  can depend only linearly on $z_{n}$. However, $f_{0}$ is a singularity and
  hence its expansion has to involve a monomial of the form $z_{n} z_{i}$, for
  some $i<n$ (cf. {\cite[Korollar 1.6]{Sai71}}). It is easy to see that
  $f_{0}$ can be brought to the form $\tilde{f}_{0} ( z_{1} , \ldots ,
  \hat{z}_{i} , \ldots , \hat{z}_{n} ) +z_{i} z_{n}$ by a biholomorphism that
  does not violate the {\tmname{wsqh}} type of $f$. Hence, by this
  biholomorphism, $f$ is transformed to a {\tmname{wsqh}} function $\tilde{f}$
  with principal part of the form $\tilde{f}_{0} ( z_{1} , \ldots ,
  \hat{z}_{i} , \ldots , \hat{z}_{n} ) +z_{i} z_{n}$, which is the one assumed
  in (\ref{w5.5}). Applying Lemma \ref{le1} to $\tilde{f}$ we can reduce it to
  $( \tilde{f}_{0} + \breve{f}' ) ( z_{1} , \ldots , \hat{z}_{i} , \ldots ,
  \hat{z}_{n} ) +z_{i}^{2} +z_{n}^{2}$, where $\tilde{f}_{0} + \breve{f}'$ is
  a {\tmname{wsqh}} function of type $( 1;l_{1} , \ldots , \hat{l}_{i} ,
  \ldots ,l_{n-1} , \hat{l}_{n} )$ or the zero function for $n=2$, in which
  case $\mathcal{L}_{0} ( f ) =1= \tfrac{1}{l_{\min}} -1$. In the former case,
  we easily see that the replacement of $f$ with $\tilde{f}_{0} + \breve{f}'$
  again does not affect the numbers $\mathcal{L}_{0} ( f )$ and $l_{\min}$.
  Thus we reduce the number of variables. Continuing in this way, either we
  end up with a sum of squares at some stage, in which case $\mathcal{L}_{0} (
  f ) =1= \tfrac{1}{l_{\min}} -1$, or with a {\tmname{wsqh}} function
  $\bar{f}$ with weights lying in the interval $\left( 0, \frac{1}{2}
  \right]$. It is easy to see that such weights have to be rational, so in
  this last case $\bar{f}$ is in fact a {\tmname{sqh}} function. Moreover, for
  such weights the multiset $L^{-}$ is empty and $l_{\min}$ is just the
  minimal of the weights. By formula (\ref{w5}), we get the desired equality.
\end{proof}

\begin{commentary*}
If $f_{0}$ is the principal part of a
{\tmname{wsqh}} function $f$, Theorem \ref{th4} implies that $\mathcal{L}_{0}
( f ) =\mathcal{L}_{0} ( f_{0} ) < \infty$. This once again signifies that a
{\tmname{wsqh}} function $f$ is automatically an isolated singularity, which
-- by the above proof and Corollary \ref{cor2.5} -- is stably (and even
{\tmem{biholomorphically}}) equivalent to a {\tmname{sqh}} function.
Similarly, one can show that $\mu_{0} ( f ) = \mu_{0} ( f_{0} )$.
\end{commentary*}

For a {\tmname{wqh}} singularity $f$ one can compute its {\L}ojasiewicz
exponent in any system of coordinates. Namely, we have:

\begin{twier}\label{th5}
Let $f: ( \mathbbm{C}^{n} ,0 ) \rightarrow ( \mathbbm{C},0
)$ be an isolated singularity such that
\[ f=g_{1}  \frac{\partial f}{\partial z_{1}} + \ldots +g_{n}  \frac{\partial
   f}{\partial z_{n}} , \]
{\noindent}for some $g_{1} , \ldots ,g_{n} \in \mathbbm{C} \{ z_{1} , \ldots
,z_{n} \}$. Let $[ \alpha_{1} , \ldots , \alpha_{n} ] \subset \mathbbm{C}$ be
the multiset of~all~eigenvalues of the matrix $\left. \frac{\partial ( g_{1} ,
\ldots ,g_{n} )}{\partial ( z_{1} , \ldots ,z_{n} )} \right|_{z=0}$. Put
$l_{j} \assign \tmop{Re}   \alpha_{j}$, $j=1, \ldots ,n$, and define $L,
L^{-} ,  L^{0} ,l_{\min}$ as in Theorem \ref{th4}. Then
$\mathcal{L}_{0} ( f ) = \frac{1}{l_{\min}} -1$.
\end{twier}

\begin{proof}
  Since $f$ is a singularity, {\cite[Lemma 4.2]{Sai71}} implies that $g_{1} (
  0 ) = \ldots =g_{n} ( 0 ) =0$. Next, {\cite[Korollar 3.3]{Sai71}} asserts
  that there exists a formal system of coordinates in which $f$ is a
  {\tmname{wqh}} formal singularity $\tilde{f}$ of type $( 1; \alpha_{1} ,
  \ldots , \alpha_{n} )$. (Here we allow the weights to be complex numbers and
  $\tilde{f}$ to be a formal power series.) But then, in this system of
  coordinates, $\tilde{f}$ is also {\tmname{wqh}} of type $( 1;l_{1} , \ldots
  ,l_{n} )$. Since $\tilde{f}$ is finitely determined, we can assume it is a
  polynomial and by {\tmname{Artin}}'s Theorem {\cite{Art68}} -- that it is
  actually biholomorphically equivalent to $f$. The theorem follows upon
  applying Theorem \ref{th4}.
\end{proof}

\begin{remark*}
Using (\ref{w2}) one can define the {\L}ojasiewicz
exponent also for $f \in \mathbbm{C} [ [ z_{1} , \ldots ,z_{n} ] ]$. With this
definition, all the above results on $\mathcal{L}_{0}$ are true in the formal
setting.
\end{remark*}

 \begin{example}
  Let $f \assign xz+xyz^{2} +xy^{3} +y^{3} z^{2} +y^{5} +y^{2} z^{4} +z^{8}$.
  Using a computer algebra system one can check that $f \in ( \nabla f )$ and
  (non-unique) eigenvalues of the Jacobian matrix~of $( g_{1} ,g_{2} ,g_{3} )$
  at $0$ are $\seqsplit{[ -2919603413161694054973386275881695714936.56 \ldots
  ,1/5,2919603413161694054973386275881695714937.56 \ldots ]}$. By Theorem
  \ref{th5}, $\mathcal{L}_{0} ( f ) =4$. We also remark that in this case the
  value of $\mathcal{L}_{0} ( f )$ can be computed using either Theorem
  \ref{th2} (because $f \in \text{SQH} \left( 1; \frac{1}{2} , \frac{1}{5} ,
  \frac{1}{2} \right)$) or {\cite[Thm.~1.8 (1\tmrsup{o})]{Ole13}}, because $f$
  is Kouchnirenko non-degenerate and its Newton diagram consists of
  exceptional faces only (see {\tmem{op.~cit.}} for the details).
\end{example}

We end the paper with a corollary concerning the conjecture of
{\tmname{Teissier}}. Namely, just as there is the famous {\tmname{Zariski}}
problem on multiplicity, the same question can be asked for the {\L}ojasiewicz
exponent: is it a topological invariant of a singularity? Since the question
seems to be very difficult (although it is known to have the affirmative
answer in case of germs of two indeterminates, see {\cite{Tei77a}} or
{\cite{Plo01}}, and also for {\tmname{qh}} singularities of three
indeterminates {\cite[Corollary 2]{KOP09}}), it is natural to ask a weaker
one:

  \begin{question*}[\tmstrong{Teissier's conjecture}]
  If $( f_{s} )$ is a topologically
  trivial deformation of a singularity $f_{0}$, does $\mathcal{L}_{0} ( f_{0}
  ) =\mathcal{L}_{0} ( f_{s} )$, for small $s \in \mathbbm{C}$?
  \end{question*}

{\noindent}It should be noted that it is already known that {\L}ojasiewicz
exponent is lower semi-continuous in $\mu$-constant families ({\cite{Tei77a}},
{\cite{Plo10}}).

For {\tmname{sqh}} (and also {\tmname{wsqh}}) singularities we can answer
Teissier's question in the affirmative.

  \begin{corollary}
  \label{cor4}If $f: ( \mathbbm{C}^{n} ,0 ) \rightarrow ( \mathbbm{C},0 )$ is
  a weakly semiquasihomogeneous function then the {\L}ojasiewicz exponent
  $\mathcal{L}_{0}$ is constant on every topologically trivial deformation of
  $f$.
\end{corollary}

\begin{proof}
  Assume first that $f$ is {\tmname{sqh}}. From Theorem \ref{th3} it follows
  that the {\L}ojasiewicz exponent of $f$ is determined by its weights. On the
  other hand, the Theorem of {\tmname{Varchenko}} {\cite{Var82}} implies that
  the weights are invariant in $\mu$-constant deformations of $f$.
  
  If we assume only that $f$ is {\tmname{wsqh}}, then we may biholomorphically
  transform $f$ into a {\tmname{sqh}} function $\tilde{f}$ (cf.~the commentary
  after Theorem \ref{th4}). This transformation also~carries any deformation
  $f_{s}$ of $f$ to a deformation $\tilde{f}_{s}$ of $\tilde{f}$; and the
  Milnor and~{\L}ojasiewicz numbers remain unchanged. Thus, we reduce the
  problem to the first case.
\end{proof}

{\small{\begin{flushright}
  Szymon Brzostowski{
  
  }Faculty of Mathematics{
  
  }and Computer Science{
  
  }University of {\L}{\'o}d{\'z}{
  
  }ul. Banacha 22,{
  
  }90-238 {\L}{\'o}d{\'z}, Poland{
  
  }brzosts@math.uni.lodz.pl
\end{flushright}}}

\end{document}